\RequirePackage[l2tabu,orthodox,abort]{nag}

\documentclass{birkjour}

\usepackage{fixltx2e}
\usepackage[all,error]{onlyamsmath}

\usepackage{amssymb,amsfonts,amsmath,
amsrefs,enumerate,color,url,hyperref,mathbbol,dsfont}
\usepackage{mathrsfs}
\usepackage{tikz}

\usepackage[T1]{fontenc}
\usepackage[cp1250]{inputenc}

\usepackage[strict=true]{csquotes}

 \newtheorem{thm}{Theorem}[section]
 
 \newtheorem{lem}[thm]{Lemma}
 \newtheorem{prop}[thm]{Proposition}
 \theoremstyle{definition}
 
 \theoremstyle{remark}
 \newtheorem{rem}[thm]{Remark}
 
 \numberwithin{equation}{section}
\usepackage{mathtools}
\usepackage{enumitem}
\renewcommand{\theenumi}{\textup{(\alph{enumi})}}
\renewcommand{\labelenumi}{\theenumi}

\begin{document}

%\title[Diffusion on a thin layer]
      %{A singular perturbation leading to a model of kinase activity on a thin layer close to cell's membrane \tcm{Czy to dobry tytul?}}
			
\title[Irregular convergence]{Irregular convergence of mild solutions of semilinear equations}

%----------Author 1
\author[A. Bobrowski]{Adam Bobrowski}

\address{
Lublin University of Technology\\
Nadbystrzycka 38A\\
20-618 Lublin, Poland}

\email{a.bobrowski@pollub.pl}

\author[M. Kunze]{Markus Kunze}
\address{Universit\"at Konstanz, \\Fachbereich Mathematik und Statistik,\\ 78457 Konstanz, Germany}
\email{markus.kunze@uni-konstanz.de}

%\author[T. Lipniacki]{Tomasz Lipniacki}
%\address{Institute of Fundamental Technological Research,\\ Polish Academy of Sciences,\\ Warsaw, Poland}

%\email{tlipnia@ippt.pan.pl}

\newcommand{\cxi}{(\xi_i)_{i\in \N} }
\newcommand{\lam}{\lambda}
\newcommand{\eps}{\varepsilon}
\newcommand{\ud}{\, \mathrm{d}}
\newcommand{\pr}{\mathbb{P}}
\newcommand{\f}{\mathcal{F}}
\newcommand{\s}{\mathcal{S}}
\newcommand{\h}{\mathcal{H}}
\newcommand{\ai}{\mathcal{I}}
\newcommand{\R}{\mathbb{R}}
\newcommand{\C}{\mathbb{C}}
\newcommand{\Z}{\mathbb{Z}}
\newcommand{\N}{\mathbb{N}}
\newcommand{\Y}{\mathbb{Y}}
\newcommand{\e}{\mathrm {e}}
\newcommand{\tif}{\widetilde {f}}
\newcommand{\slam}{\sqrt {\lam}}
\newcommand{\Id}{{\mathrm{Id}}}
\newcommand{\cic}{C_{\mathrm{mp}}}
\newcommand{\cod}{C_{\mathrm{odd}}[0,1]}
\newcommand{\cev}{C_{\mathrm{even}}[0,1]}
\newcommand{\cevr}{C_{\mathrm{even}}(\mathbb{R})}
\newcommand{\codr}{C_{\mathrm{odd}}(\mathbb{R})}
\newcommand{\cez}{C_0(0,1]}
\newcommand{\fod}{f_{\mathrm{odd}}} 
\newcommand{\fev}{f_{\mathrm{even}}} 
\newcommand{\sem}[1]{\mbox{$\left (\e^{t{#1}}\right )_{t \ge 0}$}}
\newcommand{\semi}[1]{\mbox{$\left ({#1}\right )_{t > 0}$}}
\newcommand{\semt}[2]{\mbox{$\left (\e^{t{#1}} \otimes_\varepsilon \e^{t{#2}} \right )_{t \ge 0}$}}
\newcommand{\tr}{\textcolor{red}}
\newcommand{\cea}{C_A}
\newcommand{\ceat}{C_A(t)}
\newcommand{\cosinea}{(\ceat )_{t\in \R}}  
\newcommand{\sea}{S_A}
\newcommand{\seat}{S_A(t)}
\newcommand{\sema}{(\seat )_{t\ge 0}}
\newcommand{\wt}{\widetilde}
\renewcommand{\iff}{if and only if }
\renewcommand{\k}{\mathrm{k}}
\newcommand{\tcm}{\textcolor{magenta}}
\newcommand{\tcb}{\textcolor{blue}}
\newcommand{\dx}{\ \textrm {d} x}
\newcommand{\dy}{\ \textrm {d} y}
\newcommand{\dz}{\ \textrm {d} z}
\newcommand{\di}{\textrm{d}}
\newcommand{\tcg}{\textcolor{green}}
\newcommand{\lc}{\mathfrak L_c}
\newcommand{\ls}{\mathfrak L_s}
\newcommand{\grat}{\lim_{t\to \infty}}
\newcommand{\gra}{\lim_{n\to \infty}}
\newcommand{\grae}{\lim_{\eps \to 0}}
\newcommand{\rez}[1]{\left (\lam - #1\right)^{-1}}
\newcommand{\papa}{\hfill $\square$}
\newcommand{\papap}{\end{proof}}
\newcommand {\x}{\mathbb{X}}
\newcommand{\aex}{A_{\mathrm ex}}
\newcommand{\jcg}[1]{\left ( #1 \right )_{n\ge 1} }
\newcommand {\y}{\mathbb{Y}}
\newcommand{\injtp}{\x \hat \otimes_{\varepsilon} \y}
\newcommand{\pin}{\|_{\varepsilon}}
\newcommand{\mc}{\mathcal}
\newcommand{\inter}{\left [0, 1\right ]}
\newcommand{\lir}{\lim_{r \to 1}}
\newcommand{\ha}{\mathfrak {H}}
\makeatletter
\newcommand{\normt}{\@ifstar\@normts\@normt}
\newcommand{\@normts}[1]{%
  \left|\mkern-1.5mu\left|\mkern-1.5mu\left|
   #1
  \right|\mkern-1.5mu\right|\mkern-1.5mu\right|
}
\newcommand{\@normt}[2][]{%
  \mathopen{#1|\mkern-1.5mu#1|\mkern-1.5mu#1|}
  #2
  \mathclose{#1|\mkern-1.5mu#1|\mkern-1.5mu#1|}
}
\makeatother

\thanks{Version of \today}
%----------Au
% \tcm{Inne numerki?}}

%\dedicatory{}
%%% ----------------------------------------------------------------------

\begin{abstract} We prove that even irregular convergence of semigroups of operators implies similar convergence of mild solutions of the related semi-linear equations with Lipschitz continuous nonlinearity. This result is then applied to three models originating from mathematical biology: shadow systems, diffusions on thin layers, and dynamics of neurotransmitters. \end{abstract}

%%% ----------------------------------------------------------------------
\maketitle

\newcommand{\oper}{\mathfrak R_r}
\newcommand{\opern}{\mathfrak R_{\rn}^\mho}
\newcommand{\brn}{\mbox{$\Delta^\mho_{\rn}$}}
\newcommand{\bro}{\mbox{$\Delta_{\rn}$}}
\newcommand{\rn}{r}
\newcommand{\cern}{C\hspace{-0.07cm}\left[\rn, 1\right ]}
\newcommand{\cernbez}{C\left[\rn, 1\right ]}
\newcommand{\cep}{C\hspace{-0.07cm}\left[ 0, 1\right ]}
\newcommand{\copi}{C[0,\pi]}
\newcommand{\cerec}{C\hspace{-0.07cm} \left ([0,\pi]\times [r,1]\right)}
\newcommand{\cerecbez}{C \left ([0,\pi]\times [r,1]\right)}
\newcommand{\cerecdwa}{C^2\hspace{-0.07cm} \left ([0,\pi]\times [r,1]\right)}
\newcommand{\cerecj}{C\hspace{-0.07cm} \left ([0,\pi]\times \left [0 ,1\right ]\right)}
\newcommand{\xprim}{C_\theta (UR)}
\newcommand{\ie}{i.e., }
\newcommand{\rla}{R_\lambda}
\section{Introduction}\label{intro} 
A number of phenomena of mathematical biology and mathematical physics, such as reaction-diffusion equations \cite{smoller}, can be modelled by a semilinear equation in a Banach space, \ie an equation of the form 
\begin{equation} \label{slin} \frac {\ud u(t)}{\ud t} = A u(t) + F(t, u(t)),  \qquad t \ge 0, \end{equation}
where $A$ is the generator of a strongly continuous semigroup \sem{A} of operators in a Banach space $\x$,  and $F:\R^+ \times \x \to \x$ is a jointly continuous map that is globally Lipschitz continuous in the second variable:  there is an $L>0$ such that for $x,y \in \x$ and $t \ge 0$
\begin{equation}\label{lconst} \| F(t,x) - F(t,y) \| \le L \|x-y\|.\end{equation}
(This assumption may be a bit relaxed, see in particular our examples in Section \ref{przyklady} and consult  \cite{pazy} and/or \cite{knigaz}).
We should point out that it is the nonlinearity $F$ which is responsible for many characteristic phenomena which cannot occur in
a linear equation. As an example, we mention bistability of the system and the existence of homo- and hetero-clinical waves,
which are critical phenomena for signalling pathways in living cells \cite{hat2011,thin}.

This paper is devoted to the question of convergence of mild solutions of a sequence of semilinear equations. Thus, given semigroups $\sem{A_n}$ and $x \in \x$ we are interested in the question of whether the mild solutions to  the equation 
\begin{equation} \label{slinn} \frac {\ud u_n(t)}{\ud t} = A_n u(t) + F(t, u_n(t)),  \qquad u_n(0)=x, t \ge 0, \end{equation}
converge or not. We recall that a mild solution of \eqref{slinn} is a function $u_n$ satisfying
\begin{equation} \label{smild} u_n(t) = \e^{tA_n}x + \int_0^t \e^{(t-s)A_n}F(s,u_n(s))\ud s, \qquad t \ge 0, \end{equation}
and it can be proved using Banach's Fixed-Point Theorem that for each $x$ such a solution exists and is unique (see e.g. \cite{knigaz}, Chapter 36).

In addressing such a question it is sometimes a good strategy to prove convergence of the semigroups $\sem{A_n}$ first, and only then start to worry about the second term in \eqref{smild}. Here the classical convergence theorem for semigroups of operators, often called the Trotter--Kato--Neveu Theorem \cite{knigaz,engel,goldstein,pazy}, comes in handy. This theorem asserts that a sequence of semigroups $\sem{A_n}, n \ge 0,$ converges strongly (\ie for each $x$ in norm of the Banach space $\x$) to a strongly continuous semigroup if and only if: 
\begin{itemize} 
\item [(i)] There exist constants $C$ and $\omega $ such that 
\[ \|\e^{tA_n}\| \le C\e^{\omega t}, \qquad t\ge 0 , n \ge 1.\]
\item [(ii)] For $\lam > \omega $, the resolvents converge (strongly), \ie there exists the limit
\[\rla = \gra \rez{A_n}. \]
\item [(iii)] The closure of the range of the operator $\rla $ is the entire $\x$ for some (and hence for all) $\lam >\omega. $
\end{itemize}

In this case, the operators $R_\lambda, \lam >\omega $ form the resolvent of an operator $A$, \ie $R_\lambda = (\lambda - A)^{-1}$
for $\lambda >\omega$, and $A$ is the generator of the limit semigroup.
We should point out that, if the conditions of the Trotter--Kato--Neveu Theorem are met, the semigroups converge uniformly on the compact subsets of $[0,\infty)$. In this situation,  a straightforward use of the Gronwall lemma implies that the mild solutions of \eqref{smild} converge uniformly on compact subsets of $[0,\infty)$ to mild solutions of \eqref{slin}.

However, this is not the entire story \cite{knigaz}, for quite often, in singular perturbations for example, the third condition in the Trotter--Kato--Neveu Theorem fails. Then, the limit semigroup acts on the regularity space $\x_0$ defined as the closure of the (common) range of $\rla$: 
\begin{equation} \label{xzero} \x_0 = \text{cl\,} Range \rla \end{equation}
 and for $x\in \x_0$ convergence is still uniform on the compact subsets of $[0,\infty)$; such convergence is termed regular, and therefore $\x_0$ is called the regularity space. There are examples showing that for $x\not \in \x_0$, $\gra \e^{tA_n}x$ in general does not exists (see \cite{deg} or \cite{knigaz}, Chapter 7). Nevertheless, in many important cases it does, except that convergence is not uniform on compact subsets of $[0,\infty)$, but merely uniform on compact subsets of $(0,\infty)$ which is to say that it is uniform in each interval of the form $[\tau^{-1}, \tau], \tau >1$ (see \cite{note} or \cite{knigaz}, Chapter 28). Such convergence is termed irregular. 

Thus, in this paper, motivated in particular by biological models presented in our Section \ref{przyklady}, we consider the following situation. 

\begin{itemize} 
\item We assume that the semigroups $\sem{A_n}$ are equibounded, \ie that condition (i) in the Trotter--Kato--Neveu Theorem holds with $\omega =0$. This is a customary simplification in semigroup theory, since the general case may be deduced from this one by considering $A_n' = A_n + \omega I.$ 
 
\item We assume that there is a closed subspace $\x_0$ (the regularity space) and a strongly continuous semigroup \sem{A} in $\x_0$, such that for $x \in \x_0$, \[ \gra \e^{tA_n}x = \e^{tA}x \] uniformly on the compact subsets of $[0,\infty)$. (In fact, $\x_0$ is given by \eqref{xzero}.) 
\item There is also a projection $P: \x \to \x_0$ such that for $x \not \in \x_0$,  \[ \gra \e^{tA_n}x = \e^{tA}Px, \qquad t >0 \] uniformly on the compact subsets of $(0,\infty)$.  
\end{itemize}

We note that in choosing these assumptions, besides motivations of biological origin named above, we were guided by T.G. Kurtz's singular perturbation theorem \cite{knigaz,ethier,kurtzper,kurtzapp}, where assumptions of similar nature are made. % (see also examples section, Section \ref{przyklady}). 
Our main theorem reads:

\begin{thm}\label{main} Under the stated assumptions, the mild solutions of \eqref{slinn} with initial value $x \in \x$ converge to that of 
\begin{equation} \label{slinp} \frac {\ud u(t)}{\ud t} = A u(t) + PF(t, u(t)),  \qquad t \ge 0, \end{equation}
(note the change in the non-linear term) with initial value $Px$. As in the case of semigroups, for $x\in \x_0$, the convergence is uniform on compact subsets of $[0,\infty)$, and for $x\not \in \x_0$ it is uniform on compact subsets of $(0,\infty)$. \end{thm}

It should perhaps be remarked here that the theorem applies in particular to the situation where 
\( F(t,x) = B(t)x \)
and $t \mapsto B(t)$ is a strongly continuous family of equibounded operators. In the yet more concrete case where $B$ does not depend on time $t$, this theorem has been obtained in \cite{knigaz}*{Chapter 29} and its implications for the so-called Stein model of neuronal activity were discussed in Chapter 30 there.

The formulation of our main result presented above is best suited for the motivating models but, remarkably, our proof never uses the fact that the semigroups involved are strongly continuous at $t=0$. We can thus consider equibounded, but perhaps not continuous at $t=0$,  semigroups
$\semi{T_n(t)}$, $n\ge 1$ and $\semi{T(t)}$. In order to generalize the variation of constants formula \eqref{smild} to this 
setting, strong Bochner measurability of the orbits seems a minimal assumption. As is well known (see \cite[Theorem 10.2.3]{hp})
this already implies strong continuity on $(0,\infty)$. We will thus assume that all semigroups involved have orbits which are continuous on $(0,\infty)$. As we shall see, 
for each $x \in \x$, $n \ge 1$ there exists precisely one continuous function $u_n: (0,\infty) \to \x$ of at most exponential growth, bounded on $(0,1]$,  and such that 
\begin{equation}\label{ueny} u_n (t) = T_n(t)x + \int_0^t T_n(t-s)F(s,u_n(s))\ud s , \qquad t >0,  \end{equation}
and of course the same may be said of existence and uniqueness of $u$ satisfying
\begin{equation}\label{uu} u(t) = T(t)x + \int_0^t T(t-s)F(s,u(s))\ud s , \qquad t >0.  \end{equation}
In this context our main theorem may be rephrased as follows.

\begin{thm}\label{mainprim} Suppose 
\begin{equation}\label{jedynka} \gra T_n(t) x = T(t) x , \qquad x \in \x, t >0.\end{equation}
Then $u_n$ converge to $u$ uniformly on the compact subsets of $(0,\infty)$. Additionally, if $T(0)$ and $T_n(0)$ are defined for all $n$ and $x$ is such that \eqref{jedynka} is also true for $t=0$ and the limit is uniform on the compact subsets of $[0,\infty)$, 
then also is the corresponding mild solutions $u_n$ converge uniformly on the compact subsets of $[0,\infty)$ to $u$. \end{thm}

We then obtain Theorem \ref{main} as a direct consequence of Theorem \ref{mainprim}, setting $T_n(t)=\e^{tA_n},$ and $T(t) \coloneqq \e^{tA}P$.
The latter theorem will be proved in Section \ref{dowod}. The remaining part of the paper is devoted to applications, including shadow systems, diffusion in a thin layer between two parallel planes, and modeling activity of fast neurotransmitters:
these are presented in Section \ref{przyklady}.

\section{Proof of Theorem \ref{mainprim}}\label{dowod}

\newcommand{\bcl}{\mbox{$BC_\lam(\R^+_0,\x)$ }}
\newcommand{\elj}{\mbox{$L^1\left ((0,\tau),\x \right)$ }}
\newcommand{\wroc}{\hspace{-0.1cm}}

Whenever we have a sequence of equations that can be solved by means of the Banach Fixed-Point Theorem, there 
is a `folklore recipe' to establish convergence of the solutions. All one needs to establish is \begin{itemize} 
\item [(i) ] a uniform bound
on the Lipschitz constants of the fixed-point maps less than 1 and 
\item [(ii) ] pointwise convergence of the fixed-point maps;
\end{itemize} 
see Section \ref{sect.l1}, Proposition \ref{folk}, below for a more precise statement.
While this general principle does not appear so often in the study of deterministic equations (where it seems that 
the use of Gronwall's Lemma is preferred), it is used quite often in the context of stochastic equations, see 
\cite{b97, mpr10, kvn11}.

As so often, the main difficulty in following the path described above is the choice of the `right' underlying space. This is also the case in our situation. While, under assumption of strong continuity of all approximating semigroups,  it is possible to use a space of continuous functions on $[0,\infty)$ to solve
a \emph{single} equation, there is no hope that this space will be suitable for establishing a convergence result,
as the maps involved do not converge (uniformly on compact subsets of $[0,\infty)$). Thus, to prove
Theorem \ref{mainprim}, we proceed in two steps. First, rather than a space of continuous functions, we consider a space of 
integrable functions as our underlying space and establish convergence of our solutions in the $L^1$-sense. It is only afterwards that
this result is refined to uniform convergence on compact subsets of $(0,\infty)$. In Section \ref{uniform}, we will see in fact that while convergence of $T_n(t)x $ to $T(t)x$ may perhaps be non-uniform in the vicinity of $t=0$, the other term in \eqref{ueny} converges uniformly in any compact subset of $[0,\infty)$.

\subsection{Existence of the functions \texorpdfstring{$u$}{} and \texorpdfstring{$u_n$}{}}\label{existence}

In this section, given $x\in \x$, we establish existence and uniqueness of $u$ satisfying \eqref{uu} (in the sense specified below); this argument certainly applies for $u_n$, as well. We focus on global existence and for this we assume additionally that there are non-negative constants $C_0$ and $\lam_0$  such that 
\begin{equation} \label{exp} \int_0^t \|F(s,0)\| \ud s \le C_0\e^{\lam_0 t}, \qquad t > 0; \end{equation}
we stress that for local existence, \ie for existence on finite intervals this assumption is not necessary and that it is automatically satisfied if $F$ does not depend on $t$ (which is the case in all our examples). Recall that besides Lipschitz continuity described in  \eqref{lconst} we assume that for certain $C$,
\begin{equation} \|T_n(t)  \| \le C, \qquad t >0 \label{wspol} \end{equation}
and thus also $\|T(t)\| \le C$ for $t>0$.

For $\lam \ge \lam_0$, let \bcl be the space of continuous functions $v$ on $\R^+_0$ with values in $\x$, which are bounded on $(0,1]$ and are of at most exponential growth $\e^{\lambda t}$, \ie satisfy 
\begin{equation}\label{vlam} \|v\|_\lam := \sup_{t >0 } \e^{-\lambda t} \| v(t) \| < \infty .\end{equation} 
When equipped with the norm $\|\cdot \|_\lam $, \bcl is a Banach space. (This type of norm has been first introduced by Adam Bielecki, see \cite{abielecki,edwards}.)
Next, for $v \in$ \bcl \hspace{-0.1cm}, we introduce 
\[ [\Phi (v)](t) = T(t)x + \int_0^t T(t-s)F(s,v(s))\ud s ,\qquad  t \ge 0.\]
It is clear that $\Phi (v) $ is a continuous function on $(0,\infty)$ which on $(0,1]$ is bounded by
\begin{align*} 
&\phantom{\le }\, \, \, C\|x\| + \int_0^t\|T(t-s)[F(s,v(s)) - F(s,0)]\| \ud s + \int_0^t\|T(t-s)F(s,0)]\| \ud s \\
&\le C\|x\| + C \int_0^t\|F(s,v(s)) - F(s,0)\| \ud s +  C\int_0^t\|F(s,0)\| \ud s \\
&\le C\|x\| + C L \max_{s \in (0,1]} \|v(s)\| + C\int_0^1 \|F(s,0)\| \ud s < \infty; 
\end{align*}
the last integral being finite since $F$ is assumed to be jointly continuous.  Also, taking another element of \bcl\hspace{-0.1cm}, say $w$, we have for $t>0$,
\begin{align} \e^{-\lam t}  \|[\Phi (v)] (t) -  [\Phi (w)](t)\| &\le \e^{-\lam t} \int_0^t\|T(t-s)[F(s,v(s)) - F(s,w(s))]\| \ud s \nonumber \\
&\le CL \int_0^t \e^{-\lam (t-s)} \e^{-\lam s} \|v(s) - w(s)\| \ud s \nonumber  \\
&\le CL \int_0^t \e^{-\lam (t-s)}  \ud s   \|v - w\|_\lam \nonumber  \\
& < \frac {CL}\lam \|v - w\|_\lam. \label{rachuny}
 \end{align}
For $w=0$, we have
\[ \|[\Phi (w)] (t)\| \le C\|x\| + C \int_0^t\|F(s,0)\| \ud s, \] 
proving, by assumption \eqref{exp}, that $\Phi (0)$ belongs to all \bcl \hspace{-0.1cm}, $\lam \ge \lam_0 $. 
Therefore, \eqref{rachuny} with  $w=0$ shows that $\Phi (v) $ belongs to \bcl as well: $\Phi$  maps \bcl into itself.  Moreover,  taking the supremum over $t >0$ we obtain 
\[ \|\Phi (v) - \Phi (w)  \|_\lam \le \frac {CL}\lam \|v - w \|_\lam   .\]
It follows that in the spaces \bcl with $\lam > CL $ (and $\lam \ge \lam_0$), $\Phi$ is a contraction. The Banach Fixed-Point Theorem shows now existence and uniqueness of a fixed-point of $\Phi$, \ie existence and uniqueness of a $u$ satisfying \eqref{uu}. 

One may wonder if the solution thus obtained is also locally  unique, but such uniqueness can be ascertained by the same calculation in the space $BC\left ((0,\tau],\x \right ), \tau >0$ of bounded, continuous functions on $(0,\tau] $ with values in $\x$ equipped with the Bielecki-type norm 
\[ \|v\|_\lam = \sup_{t \in (0,\tau]} \e^{-\lam t} \|v(t)\| . \]
We stress again that for existence (and uniqueness) of local solutions assumption \eqref{exp} is not necessary.

\subsection{Convergence in the \texorpdfstring{$L^1$}{} sense}\label{sect.l1}

Our goal in this subsection is to prove that for any $\tau >0$
\begin{equation}\label{l1} \gra \int_0^\tau \|u_n(s) - u(s) \| \ud s .\end{equation}
We will use the following `folk wisdom' theorem. 

\begin{prop} \label{folk} Let $(S,d)$ be a complete metric space and suppose that maps $\Phi_n: S\to S$, $n \ge 1$ are Lipschitz continuous with the same Lipschitz constant $q \in (0,1)$, \ie for all $s,s'\in S$ and $n\ge 1$,
\begin{equation}\label{g1} d(\Phi_n(s),\Phi_n(s'))\le q \, d (s,s'). \end{equation}
Assume also that for each $s \in S$ the limit
\begin{equation}\label{g2} \Phi(s) := \gra \Phi_n(s) \end{equation}
exists. Then the unique fixed-points $s_n^*$ of maps $\Phi_n$ (which exist by the Banach Fixed-Point Theorem) converge to an $s^*\in S$ which is the unique fixed-point of the map $\Phi$.
 \end{prop} 

We present two proofs of this result; the first is by straightforward calculation, and the second is almost by inspection. The advantage of the latter, though a bit longer, argument is that it shows that the proposition is a direct consequence of the Banach Fixed-Point Theorem. 

\begin{proof}[Proof by calculation] Combing \eqref{g1} and \eqref{g2} we see that $\Phi$ is Lipschitz continuous with Lipschitz constant $q$, and thus possesses  precisely one fixed-point. Call it $s^*$. We have 
\begin{align*} 
d(s^*,s_n^*) &\le d (s^*,\Phi_n(s^*)) + d(\Phi_n(s^*),s_n^*) \\
&=  d (\Phi(s^*),\Phi_n(s^*)) + d(\Phi_n(s^*),\Phi_n(s_n^*)) \\
&\le   d (\Phi(s^*),\Phi_n(s^*)) + q \, d(s^*,s_n^*).\end{align*}
Thus, 
\[ d(s^*,s_n^*) \le \frac 1{1-q} d (\Phi(s^*),\Phi_n(s^*)),\]
and the right-hand side converges to $0$, by assumption. \end{proof}

\begin{proof}[Proof by inspection] Let $c(S)$ be the space of convergent sequences of elements of $S$; this is a complete metric space with metric given by \[d_{\text{sup}} (\jcg{s_n},\jcg{s_n'} ) = \sup_{n \ge 1} d(s_n,s_n').\] If $\jcg{s_n}$ converges in $S$ then, by \eqref{g1} and \eqref{g2}, so does $\jcg{\Phi_n(s_n)}$ and we have 
\begin{equation}\label{g3} \gra \Phi_n(s_n) = \Phi (\gra s_n ). \end{equation}
Therefore, $\jcg{s_n} \mapsto \jcg{\Phi_n(s_n)} $ defines a map, call it $\Psi$, in $c(S)$, and \eqref{g1} shows that $\Psi$ is Lipschitz continuous with the constant $q$. The Banach Fixed-Point Theorem asserts now that $\Psi$ has a unique fixed-point. By the very definition it is clear that the $n$-th coordinate of this fixed-point is $s_n^*$, \ie we have 
\[ \jcg{\Phi_n(s_n^*)} = \Psi \jcg{s_n^*} = \jcg{s_n^*} .\] 
The limit $s^*=\gra s_n^*$ exists since the fixed-point belongs to $c(S)$. Using \eqref{g3} we obtain \( \Phi(s^*) = s^*\), proving that $s^*$ is a fixed-point of $\Phi$. On the other hand, combing \eqref{g1} and \eqref{g2} we see that $\Phi$ is Lipschitz continuous with Lipschitz constant $q$, and thus possesses  precisely one fixed-point. \end{proof}

Returning to the proof of \eqref{l1}, we fix $\tau>0$ and consider the space \elj of (equivalence classes of) Bochner measurable functions on $(0,\tau)$ with values in $\x$ that are integrable on this interval. This is a Banach space when equipped with the norm \[ \|v\|_{L^1} =\int_0^\tau \|v(s)\| \ud s \] but for us it will be convenient to work with the equivalent Bielecki-type norm 
\begin{equation}\label{btvl} \|v\|_\lam = \sup_{t \in (0,\tau]} \e^{-\lam t} \int_0^t \|v(s)\| \ud s \end{equation} 
where $\lam >CL$. For $v \in \elj$let 
\[ [\Phi_n (v)] (t) = T_n(t)x + \int_0^t T_n(t-s)F(s,v(s))\ud s , \qquad t \in (0,\tau);\] $\Phi_n (v)$ is a measurable function. Since 
\[ \|[\Phi_n (0)] (t) \| \le C\|x\| + C\int_0^t \|F(s,0)\| \ud s \le C\|x\| + C\int_0^\tau \|F(s,0)\| \ud s , \]
$\Phi_n (0)$ is a bounded function, and thus a member of \elj \hspace{-0.1cm}. Also, for $w \in \elj$and $t >0$,
\begin{align} %\e^{-\lam t} 
\nonumber
\int_0^t \|[\Phi_n(v)] (s) &- [\Phi_n (w)](s)\|\ud s \\ &\le 
%\e^{-\lam t} 
\int_0^t \int_0^s \left \|T_n(s-u) [F(u,v(u)) - F(u,w(u))] \right \| \ud u \ud s \nonumber \\
&\le CL\int_0^t \int_0^s \left \|v(u) - w(u)\right \| \ud u \ud s \label{achun}. 
\end{align} 
Since for $w=0$ the right-hand side does not exceed $CL\tau \|v\|_{L^1}$, $\Phi_n (v)$ is a member of \elj \hspace{-0.1cm}. Thus $\Phi_n$  maps \elj into itself. 
Multiplying the right-most and left-most members of \eqref{achun} by $\e^{-\lam t}$ we see also that
\[ \|\Phi_n (v) - \Phi_n(w)\|_\lam \le CL \sup_{t\in (0,\tau]} \int_0^t \e^{-\lam (t-s)}\ud s \|v-w\|_\lam < \frac {CL}{\lam } \|v-w\|_\lam .\] 
It follows that the maps $\Phi_n$ are Lipschitz continuous with the same Lipschitz constant $\frac {CL}{\lam} \in (0,1)$ as in our Proposition \ref{folk}. In particular, each of them has a unique 
fixed-point. This fix point, say $u_n$, satisfies \eqref{ueny} almost surely for $t\in (0,\tau].$ Since from the previous subsection we know that solutions of \eqref{ueny} are unique in $BC\left ((0,\tau],\x \right)$ and since members of $BC\left ((0,\tau],\x \right)$ belong to $\elj$as well, the so-found solutions of \eqref{ueny} coincide with those found earlier. Hence, by Proposition \ref{folk},  \eqref{l1} will be proved once we show that 
\[\gra \Phi_n(v) = \Phi(v) \] 
in the sense of $\|\cdot \|_\lam $ norm or, equivalently, in the norm $\|\cdot \|_{L^1},$ where, of course, 
\[ [\Phi(v)](t) = T(t)x + \int_0^tT(t-s) F(s,v(s)) \ud s .\]

That $\gra \int_0^\tau \|T_n(t)x - T(t)x\|\ud s =0$ is clear by the Dominated Convergence Theorem, because the integrands here converge pointwise and are dominated by $2C\|x\|$. Hence, it suffices to show that 
\begin{equation}\label{h1}
\gra \int_0^\tau \int_0^t \|T_n(t-s) F(s,v(s)) - T(t-s) F(s,v(s)) \| \ud s \ud t = 0. \end{equation}  
The integrands in 
\begin{equation}\label{h2} 
\int_0^t \|T_n(t-s) F(s,v(s)) - T(t-s) F(s,v(s)) \| \ud s \end{equation}
converge pointwise to $0$ and are bounded by 
\begin{align}
2C\|F(s,v(s))\| &\le 2C\|F(s,v(s)) - F(s,0)\| + 2C\|F(s,0)\| \nonumber \\
&\le 2CL\|v(s)\| + 2C\sup_{0\le s\le \tau} \|F(s,0)\| \nonumber \\
&=: 2CL\|v(s)\|  + C_1. \label{elek} \end{align}
Since $s\mapsto 2CL\|v(s)\|  + C_1$ is integrable on $(0,\tau)$, invoking the Dominated Convergence Theorem again we see that the expression in \eqref{h1} converges to $0$ for all $t \in (0,\tau).$ Moreover, 
\(
\int_0^t \|T_n(t-s) F(s,v(s)) - T(t-s) F(s,v(s)) \| \ud s  \) does not exceed 
\[ 2CL \int_0^t \|v(s)\| \ud s + 2C_1 t \le 2CL \| v\|_{L^1} + 2C\tau \]
so that the same argument proves \eqref{h1}.
\subsection{Uniform convergence on compact sets}\label{uniform} 
To complete the proof of Theorem \ref{mainprim} it suffices to show that 
\[ \gra \int_0^t T_n(t-s)F(s,u_n(s))\ud s = \int_0^t T(t-s)F(s,u(s))\ud s \]
uniformly on $[0,\tau], \tau >0.$  However, the norm of the difference between the two integrals featuring here does not exceed 
\begin{align*} 
&\int_0^t \left \|T_n(t-s) [F(s,u_n(s))- F(s,u(s))] \right\| \ud s + I_n(t) \\
& \le CL \int_0^\tau \|u_n(s) - u(s)\| \ud s + I_n(t),  
\end{align*}
where $I_n(t)$ is the integral appearing in \eqref{h2} with $v$ replaced by $u$. Hence, by \eqref{l1}, we are left with proving uniform convergence of $I_n(t), t \in [0,\tau]$ to $0.$ 

We begin by recalling that by assumption for any $y\in \x$,  $T_n(r)y $ converges to $T(r)y$ uniformly on compact subsets of $(0,\infty)$. In other words, for each $y$ and each $\delta >0$, 
\[ \sup_{r \in [\delta, \tau]} \|T_n(r)y - T(r)y\| \]
may be made arbitrarily small by choosing sufficiently large $n$. A familiar $3$-epsilon argument (using the assumption of equiboundedness \eqref{wspol}) shows that for any compact set $K \subset \x$, the same is true of 
\[ \sup_{y\in K} \sup_{r \in [\delta, \tau]} \|T_n(r)y - T(r)y\| .\]
Now, since $u$ is bounded on $(0,\tau]$, our estimate \eqref{elek} (again, with $v$ replaced by $u$) shows that the integrands in $I_n(t)$ are dominated by a certain common constant.  Therefore, given $\eps >0$ one may find a $\delta >0$ such that the norm of the sum of parts of $I_n(t)$ resulting from integration over $[0,\delta]$ and $[t -\delta, t]$ is smaller than $\frac \eps 2,$ regardless of the choice of $\tau \ge t> 2\delta$. (If $t< 2\delta,$ $\|I_n(t)\|<\eps .$)  
Finally, the remaining integral, the one resulting from integration over $[\delta, t- \delta]$,
can by estimated by 
\begin{align*} &\int_\delta^{t-\delta} \sup_{y\in K_\delta} \|T_n(t-s)y - T_n(t-s)y\| \ud s \\ &\le (t-2\delta)\sup_{y\in K_\delta} \sup_{r \in [\delta,\tau]} \|T_n(r)y -T(r)y \|\\
&\le \tau \sup_{y\in K_\delta} \sup_{r \in [\delta,\tau]} \|T_n(r)y -T(r)y \|, \end{align*}
where $K_\delta$ is the image of $[\delta, \tau]$ via  $s \mapsto F(s,u(s))$. Since the latter function is continuous (by continuity of $u$ and joint continuity of $F$), $K_\delta$ is a compact set, and so the last expression may be made smaller than $\frac \eps 2$ by taking $n$ sufficiently large. This completes the proof.

\section{Examples}\label{przyklady}

\newcommand{\grubex}{\mathbb X}
\newcommand{\Jcg}[1]{\left ( #1 \right )_{i=1,...,N}} 

\subsection{Shadow systems}

\newcommand{\dod}{
For the analysis  of the first example in this section, we need the singular perturbation theorem due to T.G. Kurtz \cite{ethier,kurtzper,kurtzapp} (see also \cite{knigaz}). Hence, we briefly recall it here. 

Let $\jcg{\eps_n}$ be a sequence of positive numbers converging to $0.$ Suppose  $A_n , n \ge 1$ %= A_0 + \kappa_n (Q + B_n)$ 
are generators of equi-bounded semigroups $\sem{A_n}$ in a Banach space $\x$, and $Q$ generates a strongly continuous semigroup $\sem{Q}$ such that\begin{equation}\lim_{t\to \infty} \e^{tQ} x =: Px , \qquad x \in \x \label{idemp} \end{equation} exists. 
%(To recall, semi-groups $\sem{A_n}$ are said to be equi-bounded if there exists $M>0$ such that \(\|\e^{tA_n}\|\le M, t \ge 0, n \ge 1\); all the semigroups featuring in the models considered in this paper are Feller semigroups, and thus this condition is automatically satisfied with $M=1$.) 
Then (see e.g. \cite[p. 516]{hp}) $P$ is a bounded idempotent operator and \begin{equation} \label{kerrange} Ker \, Q = Range \, P, \qquad \overline{Range \,Q}= Ker \, P. \end{equation} Let  
\[ \x' = Range \, P . \] Also, let $\aex $ be the extended limit of $A_n, n \ge 1$ defined as a multivalued operator: 
\[ (x,y) \in \aex \text{ \iff } \exists _{x_n \in \mc D(A_n)} \gra x_n = x, \gra A_nx_n = y. \]

\begin{thm} \label{tkurtz} (\textbf {Kurtz}) \textit Let $A$ be an operator in $\x$, $D$ be a subset of its domain, and assume that 
\begin{itemize} 
\item [(a) ] for $x \in D$, $(x,Ax) \in \aex$ where $\aex$ is the extended limit of $A_n, n \ge 1,$ 
\item [(b) ] for $y $ in a core $D'$ of $Q$, $(y,Qy) \in B_{\textrm ex}$ where 
$B_{\textrm ex}$ is the extended limit of $\eps_n A_n,n \ge 1,$ 
\item [(c) ] the operator $PA $ with domain $D \cap \x' $ is closable and its closure $\overline{PA}$ generates a strongly continuous semigroup in $\x'.$ 
\end{itemize}
Then 
\begin{equation} \label{17a} \gra \e^{t A_n } x = \e^{t\overline{PA}} P x, \, \qquad x \in \x, t >0. 
\end{equation}
For $x \in \x'$ the same is true for $t=0$ as well, and the limit is uniform in compact subsets of $[0,\infty);$ for other $x $ the limit is uniform in compact subsets of $(0,\infty).$ 

\end{thm}}

%\begin{cor} \label{wnosek} Assume that  $A_n, n \ge 1 $ satisfy assumptions of Kurtz's theorem (in fact, it suffices to assume (a) and (b)). Let $B$ be a bounded linear operator such that $BA_n, n \ge 1$ are generators of equibounded semigroups and so is $BQ$. Also, assume that $\grat \e^{tBQ} =: P_B$ exists. If $P_B BA$ is closable and its closure generates a strongly continuous semigroup, then
%\[ \gra  \e^{tB_n} = \e^{t \overline{P_B BA}}P_B .\]\end{cor} 
%To recall, the limit is said to be almost uniform in $[0,\infty)$ if it is uniform in any compact subset of $[0,\infty)$; almost uniform limit in $(0,\infty)$ is defined similarly. 

%In terms of the related Cauchy problems, Kurtz's theorem says that (mild) solutions to the 
%\begin{equation}\label{uenl}  \frac {\ud u_n(t)}{\ud t} = A_n u_n(t), \qquad u_n(0) = x , \end{equation} 
%converge, under stated conditions, to those of 
%\[ \frac {\ud u(t)}{\ud t} = \overline{PA} u(t), \qquad u(0) = Px . \] 

 Let $S$ be a compact, Hausdorff space, $N$ and $N_0< N$ be two integers, and $B_i, i =1,...,N$ be the generators of conservative Feller semigroups in $C(S)$. 
Given $\kappa_n, n \ge 1$, such that $\gra \kappa_n=\infty,$ we consider the semigroups in $\grubex = (C(S))^N$, the Cartesian product of $N$ copies of $C(S)$, given by 
\[ T_n(t) \Jcg{f_i} = \left (\e^{tB_1} f_1,...,\e^{tB_{N_0}} f_{N_0},  \e^{t\kappa_nB_{N_0+1}} f_{N_0+1}, ..., \e^{t\kappa_n B_N} f_N \right ), \]
so that  the last $N-N_0$ `variables' are `fast', and the first $N_0$ are `slow'. We assume also  that the (strong) limits 
%\begin{equation} \label 
\begin{equation}\label{granicat} \lim_{t\to \infty} \e^{tB_i} f = : P_i f , \qquad f \in C(S) , \end{equation}
exist for each $N_0+1\le i \le N.$ Then, for each $n\ge 1$, we consider the following system of equations for an unknown function $t\mapsto u_n(t) = \Jcg{u_{n,i}(t) }$ with values 
in $\grubex$: 
\begin{align}
u_{n,i} (t) &= \phantom{\kappa_n} B_i u_{n,i}(t) + F_{i} (u_n), \qquad t\ge 0, i =1,...,N_0, \nonumber \\
u_{n,i}(t) &= \kappa_n B_{i} u_{n,i}(t) + F_{i}(u_n),  \qquad t\ge 0, i =N_0+1,...,N, \label{nieshadow}
\end{align}  
where $F_i: \grubex \to C(S), i=1,\dots, N$, are Lipschitz continuous maps. We will show that as $n\to \infty,$ 
mild solutions to these equations with initial condition $u_n(0)= \Jcg{f_i} \in \grubex $ converge to those of the Shadow System:
\begin{align}
u_{i}' (t) &= B_i u_{i}(t) + F_{i} (u), \quad \qquad t\ge 0, i =1,...,N_0, \nonumber \\
u_{i}'(t) &= P_i F_{i}(u),  \qquad \qquad \qquad \ t\ge 0, i =N_0+1,...,N, \label{shadow}
\end{align}
for an unknown function $t\mapsto u(t)=\Jcg{u_i(t)}$, with initial condition \begin{equation}\label{defp} u(0)=\mathcal P \Jcg{f_i}:= \left (f_1,...,f_{N_0},P_{N_0+1} f_{N_0+1},...,P_{N} f_{N}\right ).\end{equation} 
Indeed, this is a direct consequence of our Theorem \ref{mainprim}, because we have 
\begin{align*} \gra T_n(t) \Jcg{f_i} &= \left (\e^{tB_1} f_1,...,\e^{tB_{N_0}} f_{N_0},  P_{N_0+1} f_{N_0+1}, ..., P_N f_N \right ),\\
& = \e^{tA} \mc P \Jcg{f_i} \end{align*}
where $\mc P$ is defined in \eqref{defp}, and $\sem{A}$ is the strongly continuous semigroup in the Cartesian product
\[ \grubex_0 := [C(S)]^{N_0} \times \prod_{i=N_0+1}^N \text{Range\,}  P_i \]
given by 
\[ \e^{tA} \Jcg{f_i} = \left (\e^{tB_1} f_1,...,\e^{tB_{N_0}} f_{N_0},   f_{N_0+1}, ..., f_N \right ). \]

In looking for a prototype of this scheme, we go back to the paper by J.P. Keener \cite{keener}, who studied a model of activator-inhibitor reaction system  of the form: 
\begin{align}
\frac{\partial a}{\partial t}(t,x) &= d_a \frac{\partial^2 a}{\partial x^2}(t,x) + F_1(a(t,x),h(t,x)),\nonumber \\
\frac{\partial h}{\partial t}(t,x) &= d_h  \frac{\partial^2 h}{\partial x^2}(t,x)+ F_2(a(t,x),h(t,x)),\label{keener}
\end{align}
where both activator $a$ and inhibitor $h$ are distributed on $[0,1]$ and $F_1$ and $F_2$ and certain Lipschitz continuous functions. Under additional assumptions on $F_1$ and $F_2$, and under Neumann boundary conditions,  Keener studies this system 
in the case where the inhibitor diffusion coefficient $d_h$ is much larger than that for the activator (\ie $d_a$); he seems to be also responsible for introducing the term \emph{Shadow System}. Since the Laplace operator with Neumann boundary conditions is known to have property 
\eqref{granicat} with the limit projection  mapping an $f \in C([0,1])$ to the constant function being the average of $f$ over $[0,1]$, Keener's example falls into our scheme, and the Shadow System for \eqref{keener},  obtained by letting $d_h \to \infty$, is of the form 
\begin{align*}
\frac{\partial a}{\partial t}(t,x) &= d_a \frac{\partial^2 a}{\partial x^2}(t,x) + F_1(a(t,x),h(t,x)),\nonumber\\
\frac{\partial h}{\partial t}(t,x) &= \int_0^1 F_2(a(t,x),h(t,x))\ud x ;
\end{align*}
with initial conditions $a(0,x)=a_0(x), h(0,x) = \int_0^1 h_0(x) \ud x$ where $a_0$ and $h_0$ are initial conditions of \eqref{keener} (see Keener's equations (5.7a)--(5.7c)).

As a second example, we consider the following (rescaled) model of early lung cancer due to Marciniak--Czochra and Kimmel, see \cite{markim1,markim2}:
\begin{align}
\frac {\partial c}{\partial t}(t,x) &= ((2p-1)a(b(t,x),c(t,x)) - d_c) c(t,x), \nonumber \\
\frac {\partial b}{\partial t} (t,x) & = \alpha(c(t,x)) g(t,x) - d_b b(t,x) - db(t,x) , \nonumber \\
\frac {\partial g}{\partial t} (t,x) & =\frac 1\gamma \frac{\partial^2 g}{\partial x^2} - \alpha (c(t,x)) g(t,x) - d_g g(t,x) + \kappa (c(t,x)) + db(t,x), \label{mck}
\end{align}
where $c,b$ and $g$ are densities of cells, and of bound and free growth factor molecules, respectively, distributed across the unit interval $[0,1]$ (we keep the notations of \cite{markim1}). Without going into details on the meaning of functions and coefficients in this system (see also Chapter 36 in \cite{knigaz}), we simply state that as $\gamma \to 0$, \ie in the case where the growth factor molecules are diffusing very quickly, the solutions to this system converge to those of 
\begin{align}
\frac {\partial c}{\partial t}(t,x) &= ((2p-1)a(b(t,x),c(t,x)) - d_c) c(t,x), \nonumber \\
\frac {\partial b}{\partial t}(t,x) & = \alpha(c(t,x)) g(t,x) - d_b b(t,x) - db(t,x) ,  \label{mckshadow}
\\
\frac {\partial g}{\partial t} (t,x)& =-\int_0^1 [\alpha (c(t,x)) g (t,x)+ d_g g(t,x) - \kappa (c(t,x)) - db(t,x)]\ud x.  \nonumber\end{align}
%(The integral here is with respect to the spatial variable $x$.) 
This is a nearly direct consequence of our main theorem coupled with the fact that the third equation in the model is supplied with Neumann boundary conditions. The only difficulty lies in the fact that the $F_1, F_2 $ and $F_3$ functions involved here are not globally Lipschitz continuous: they are merely locally Lipschitz continuous (see e.g. \cite{knigaz,pazy} for appropriate definitions). However, it may be shown (see Chapter 36 in \cite{knigaz}) that so-called M\"uller conditions are satisfied: there are constants $C_1,C_2$ and $C_3$ such that at the boundaries of the region $\mc R$ where $c \in [0, C_1], b \in [0,C_2]$ and $g \in [0,C_3]$ the vector field formed by $F_1,F_2$ and $F_3$ points towards the region. Then, the remedy is to extend the $F_i$'s to globally Lipschitz continuous maps on the whole of $C([0,1])$ with the Lipschitz constant suitable for $\mc R$. Then, the theory already developed asserts existence and uniqueness of solutions to \eqref{mck} with modified $F_i$'s. But the force of M\"uller conditions is that solutions starting in $\mc R$ never leave $\mc R$, and so they are in fact  solutions for \eqref{mck} with original $F_i$'s. Hence, a more precise statement should read: mild solutions of \eqref{mck} converge to those of \eqref{mckshadow} provided they start in $\mc R$. 

This connection between \eqref{mck} and \eqref{mckshadow} has been made in \cite{nowaaniaprim}. The latter paper also gives more delicate information on the speed of convergence based on heat semigroup estimates to be found e.g., in  \cite[p. 25]{rothe} or \cite[Lemma 1.3]{winkler}. However, this result still seemed to suggest that the formation of the Shadow System is somehow related to special properties of the Neumann Laplacian. It was in \cite{osobliwiec} that we argued that the phenomenon occurs for a much larger class of Feller generators $B_i$, namely for those with the property that there exists a projection $P_i$ such that for certain constants $\varepsilon $ and $K$,  
\[ \| \e^{tB_i} - P_i \|\le K \e^{-\varepsilon t}, \qquad t> 0, i = N_0+1,..., N ,\]
but the reasoning presented in \cite{osobliwiec} seems to be flawed. Fortunately, our Theorem \ref{main} solves the problem, by proving, as we have seen, our conjecture even under less stringent condition \eqref{granicat}.

\subsection{Diffusion in the thin layer between two parallel planes}\label{thin} 

In modeling biological phenomena via reaction-diffusion equations one sometimes needs to account for the fact that diffusion occurs in a thin layer between two boundaries. For example, nuclei of so-called B-lymphocytes are so large that it is sensible to think of signal-transmitting kinases as diffusing on a two-dimensional sphere modeling the cell membrane rather than in a three-di\-men\-sion\-al region between the nucleus and the membrane (see \cite{dlajima,thin} for more details). A crucial issue related to such approximations is the question of how boundary/transmission conditions featuring in the three-dimensional model are transferred to the two-dimensional model. 

Here we consider a somewhat idealized situation in which the boundaries are two parallel planes lying close to each other, and show that in the limit the boundary conditions become legitimate, integral parts of the master equation. As discussed in \cite{thin}, failure to include these parts in the limit reaction term may render the limit equation incapable of capturing the true nature of biological processes.

More specifically, given $\eps>0$ we consider the thin layer 
\[
\Omega_\eps \coloneqq \{ (x,y,z)\in \R^3 : 0\leq z\leq \eps\};
\]
and are concerned with the reaction-diffusion equation in $\Omega_\eps$: 
\begin{align}
\label{eq.diff1}
\partial_t u(t,x,y,z) &= \partial_x^2u(t,x,y,z) +\partial_y^2u(t,x,y,z) +\partial_z^2u(t,x,y,z)\\
&\quad +F (u(t,x,y,z)), \nonumber \qquad (x,y,z) \in \Omega_\eps, t >0  
\end{align}
with the reaction term $F:\R \to \R$ assumed to be (globally) Lipschitz continuous.  
On the upper and lower planes $\{(x,y, \eps) : x, y \in \R\}$ and $\{(x,y,0) : x,y \in \R\}$
of $\Omega_\eps$  we impose Robin boundary conditions of the form
\begin{align}
%(\partial_\nu u(t,x,y,\eps) &=)\,  
\partial_z u(t,x,y,\eps)&=  -\eps c(x,y)u(t,x,y,\eps),   \label{eq.robin1} \\ %(-\partial_\nu u(t,x,y,0) &= )\, 
\partial_z u(t,x,y,0 )&= \eps d(x,y)u(t,x,y,0), \qquad x,y \in \R, t >0, \nonumber
\end{align}
where $c,d: \R^2 \to [0,\infty)$ are given, essentially bounded functions. %, and $\nu$ is the outward unit vector. 
As explained in \cite{dlajima,thin}, the scaling factor (i.e., $\eps$) is needed in these boundary conditions; otherwise the limit discussed below will be uninteresting. These boundary conditions describe a stochastic mechanism of removing some of the diffusing particles touching the boundaries. 

We want to study the behavior of the solutions to the problem 
\eqref{eq.diff1}--\eqref{eq.robin1} as the parameter $\eps\to 0$, i.e.\ when the thin layer collapses to a two dimensional plane. We will argue that as $\eps \to 0$ solutions to our system become more and more `flat in the $z$-direction', and thus may be thought of as functions of two variables. To see this effect more clearly, we look at $\Omega_\eps$ through a magnifying glass, by  introducing the  change of variables, $\tilde{z} = \eps^{-1}z$, which transforms $\Omega_\eps$ into \[\Omega:=\Omega_1. \]
We write $\tilde{u}(t,x,y,\tilde{z}) = u(t,x,y,\eps^{-1}z)$. A short computation shows that the reaction-diffusion equation \eqref{eq.diff1} transforms to 
\begin{align}
\partial_t \tilde{u}(t,x,y,\tilde{z}) &= \partial_x^2\tilde{u}(t,x,y,\tilde{z}) + \partial_y^2 \tilde{u}(t,x,y,\tilde{z}) + \eps^{-2}\partial_z^2 \tilde{u}(t,x,y,\tilde{z}) \nonumber \\ &\quad + F(\tilde{u} (t,x,y,\tilde{z}) ), \qquad (x,y,\tilde{z}) \in \Omega, t >0 . \label{eq.diff2}
\end{align}
while the boundary conditions \eqref{eq.robin1} are transformed to
\begin{align}
\partial_{\tilde{z}} \tilde{u}(t,x,y,1) &= - \eps^2 c(x,y)\tilde{u}(t,x,y,1) \label{eq.boundary}\\
\partial_{\tilde{z}} \tilde{u}(t,x,y,0) &= \eps^2 d(x,y)\tilde{u}(t,x,y,0) \qquad x,y\in \R, t >0. \nonumber  \end{align}
For notational simplicity we drop the tildes, and then rewrite this system as an abstract evolution equation on the space $L^2(\Omega)$, as follows:
\begin{equation}\label{eq.cauchyproblemf}
\partial_t u_\eps(t) = A_\eps u_\eps (t) + F(u_\eps (t)), \qquad u(0) = u_0
\end{equation}
where $u_\eps : [0,\infty) \to L^2(\Omega)$ and $A_\eps$ is a suitable realization of the differential operator
$\partial_x^2+\partial_y^2+\eps^{-2}\partial_z^2$, subject to the boundary conditions \eqref{eq.boundary}. The reaction term $F$, although denoted by the same letter as the function featuring in \eqref{eq.diff1}, has a slightly different meaning. Namely, for a $u\in L^2(\Omega)$ we may define 
\begin{equation}\label{ef} 
\left ( \mathsf F(u) \right ) (x,y,z) := F(u(x,y,z))\end{equation}
where $F$ on the right-hand side is the function from \eqref{eq.diff1}. Assuming that $F(0)=0$ or, more generally, that there is a $u\in L^2(\Omega)$ such that $\mathsf F(u) \in L^2(\Omega)$, we check, using the existence of a global Lipschitz constant for $F$, that \eqref{ef} defines a globally Lipschitz continuous map $L^2(\Omega)\to L^2(\Omega)$, with the Lipschitz constant inherited from $F$. In \eqref{eq.cauchyproblem} and in what follows, for simplicity of notation (and to comply with notations of Sections \ref{intro} and \ref{dowod}), we do not distinguish between $F$ and $\mathsf F$.

As discussed in Section \ref{intro}, in dealing with well-posedness and convergence of solutions to \eqref{eq.cauchyproblemf}, it is a good strategy to work first with the related problem without the nonlinear term:
\begin{equation}\label{eq.cauchyproblem}
\partial_t u_\eps(t) = A_\eps u_\eps (t), \qquad u(0) = u_0,
\end{equation}
and we will follow this path. To this end, we will establish well-posedness of the problem \eqref{eq.cauchyproblem} making use of the theory of \emph{sesquilinear forms}. We recall that if $H$ is a complex Hilbert space, a \emph{sesquilinear form}
on $H$ is a mapping $\mathfrak{a}: D(\mathfrak{a}) \times D(\mathfrak{a}) \to \C$ which is linear in the first component and antilinear in the second component. It is called \emph{symmetric}, if $\mathfrak{a}[u,v] = \mathfrak{a}[v,u]$. Note that for a symmetric form $\mathfrak{a}$ we have $\mathfrak{a}[u] \coloneqq \mathfrak{a}[u,u] \in \R$ for every $u\in D(\mathfrak{a})$. A symmetric form $\mathfrak{a}$ is called \emph{accretive} if
$\mathfrak{a}[u,u] \geq 0$ for all $u\in D(\mathfrak{a})$; it is called \emph{closed}, if $D(\mathfrak{a})$ is a Hilbert space with respect to the inner product $[u,v]_\mathfrak{a} = \mathfrak{a}[u,v] + [u,v]_H$. A sesquilinear form is called \emph{densely defined}, if $D(\mathfrak{a})$ is dense in $H$.

Given an accretive, symmetric and closed sesquilinear form $\mathfrak{a}$ that is densely defined, we can define the associated operator $A$ by setting 
\[
D(A) = \{ u\in D(\mathfrak{a}) : \exists\, f\in H : \mathfrak{a}[u,v] = [f, v]_H\,\,\forall\, v\in D(\mathfrak{a})\} 
\]
and $Au\coloneqq f$. We thus have $\mathfrak{a}[u,v] = [Au, v]_H$ for all $u\in D(A)$ and $v\in D(\mathfrak{a})$.
It is well known that if $A$ is associated to an accretive, symmetric, densely defined and closed sesquilinear form, then $-A$ is the generator of an analytic contraction semigroup on the space $H$. For more information on sesquilinear forms, we refer to \cite{ouhabaz}.

To employ this theory, we have to find a sesquilinear form $\mathfrak{a}_\eps$ such that $-A_\eps \coloneqq
\partial_x^2+\partial_y^2+\eps^{-2}\partial_z^2$ with boundary conditions \eqref{eq.boundary} is the associated operator. To this end, we make a formal computation. We multiply $-A_\eps u$ by a $\bar v\in L^2(\Omega)$, integrate over $\Omega $ and integrate by parts. This yields
\begin{align}
-\int_\Omega A_\eps u \bar{v}\, \di (x,y,z) \notag & =  
\int_\Omega \left [ \partial_x u\partial_x\bar v + \partial_y u \partial_y\bar v + \eps^{-2}\partial_zu \partial_z\bar v\, \right ] \di (x,y,z) \notag\\
&
\quad - \eps^{-2}\int_{\R^2} \partial_zu(x,y,1) \bar v(x,y,1)  \di (x,y)\notag\\
&\quad + \eps^{-2}\int_{\R^2}\partial_z u(x,y,0)\bar v(x,y,0)\,   \di (x,y)  \notag \\
& =  \int_\Omega \left [ \partial_x u\partial_x\bar v + \partial_y u \partial_y\bar v + \eps^{-2}\partial_zu \partial_z\bar v\right ] \,\di (x,y,z)\notag\\
&\quad 
+ \int_{\R^2} c(x,y)u(x,y,1)\bar v(x,y,1) \,\di (x,y) \notag \\
&\quad 
+ \int_{\R^2} d(x,y)u(x,y,0)\bar v(x,y,0) \,\di (x,y) \notag \\
&=: \mathfrak{a}_\eps[u,v]. \label{eq.a}
\end{align}
Here we have used the boundary conditions \eqref{eq.boundary} in the second step.
Our first result is as follows:

\begin{lem}\label{l.form}
%Suppose  both $c$ and $d$, besides being non-negative and essentially bounded, are members of %$L^2(\R^2)$, and 
Define $\mathfrak{a}_\eps$ by \eqref{eq.a} on $D(\mathfrak{a}) = H^1(\Omega)$. Then $\mathfrak{a}_\eps$ is an accretive, symmetric, densely defined and closed sesquilinear form on $H= L^2(\Omega)$. Consequently, if we denote the operator associated to $\mathfrak{a}_\eps$ by $-A_\eps$, then
$A_\eps$ generates an analytic contraction semigroup on $L^2(\Omega)$.
\end{lem}

\begin{proof}
Obviously, $\mathfrak{a}_\eps$ is a symmetric sesquilinear form and $D(\mathfrak{a})$ is dense in $L^2(\Omega)$.
Moreover, we have
\[
[u,u]_{\mathfrak{a}_\eps} = \mathfrak{a}_\eps[u] +\|u\|^2_{L^2(\Omega)} \geq \min\{1, \eps^{-2}\} \|u\|_{H^1(\Omega)}^2 \geq 0.
\]
This shows that $\mathfrak{a}_\eps$ is accretive and that the norm $\|\cdot\|_{\mathfrak{a}_\eps}$ can be estimated from below by a multiple of the $H^1$-norm. Using the fact that the trace is a bounded operator from $H^1(\Omega)$
to $L^2(\{ (x,y,1) : x,y \in \R\})$ (and from $H^1(\Omega)$
to $L^2(\{ (x,y,0) : x,y \in \R\})$), see, e.g., Part I Case C of \cite[Theorem 4.12]{adams}, and that $c$ and $d$ are essentially bounded, we see that the $\|\cdot\|_{\mathfrak{a}_\eps}$-norm is actually equivalent to the $H^1$-norm. This proves that $\mathfrak{a}_\eps$ is closed.
\end{proof}

\begin{rem}
If $u \in H^2(\Omega)$ satisfies the boundary condition \eqref{eq.robin1} in the weak sense and $v \in H^1(\Omega)$, then the formal computation made above is justified and we find that $-[A_\eps u, v]_{L^2(\Omega)} = \mathfrak{a}[u,v]$. This shows that the operator associated to the form $\mathfrak{a}_\eps$ is indeed a suitable realization of the differential operator $A_\eps$.
\end{rem}

Next we want to let $\eps\to 0$. To that end, we use a convergence theorem for symmetric sesquilinear forms due to Simon \cite{simon}. In this theorem, the situation is as follows. We are given an increasing sequence $\mathfrak{a}_n$ of accretive, symmetric, densely defined and closed sesquilinear forms on a Hilbert space $H$. That the sequence is increasing means that $D(\mathfrak{a}_{n+1}) \subset D(\mathfrak{a}_n)$ and $\mathfrak{a}_n[u] \leq \mathfrak{a}_{n+1}[u]$ for all $u\in D(\mathfrak{a}_{n+1})$. Let $-A_n$ be the associated operator. Then $A_n$ generates an analytic contraction semigroup $(\e^{tA_n})_{t\geq 0}$ on $H$. 
We  define the limit form $\mathfrak{a}$ by setting
\[
D(\mathfrak{a}) \coloneqq \Big\{ u\in \bigcap_{n\in \N} D(\mathfrak{a}_n) : \sup_{n\in \N} \mathfrak{a}_n[u]<\infty\Big\}
\]
and $\mathfrak{a}[u] \coloneqq \sup_{n\in \N}\mathfrak{a}_n[u]$. For $u\neq v$ the values of $\mathfrak{a}[u,v]$
are defined via polarization. It turns out that $\mathfrak{a}$ is closed, accretive and symmetric as well. However, 
$\mathfrak{a}$ need not be densely defined. Let us put $H_0\coloneqq \overline{D(\mathfrak{a})}$. Then 
$\mathfrak{a}$ is a closed, densely defined, accretive and symmetric sesquilinear form on $H_0$. Let $-A$ be the associated operator on $H_0$. Then $A$ generates an analytic contraction semigroup on $H_0$. We extend this semigroup to all of $H$, setting it $0$ on $H_0^\bot$. With slight abuse of notation, we write $\e^{tA}P$ for this extension, where $P$ is the orthogonal projection onto $H_0$. Simon's theorem asserts  that in such circumstances $\e^{tA_n}$ converges irregularly to $\e^{tA}P$ as $n\to \infty$.

%as a corollary to the Simon theorem:
In the situation of the thin layer, we have  the following result: 
\begin{thm}
Consider the forms $\mathfrak{a}_\eps$ defined via \eqref{eq.a} on $D(\mathfrak{a}_\eps)= H^1(\Omega)$. This family increases (as $\eps$ decreases) and the limit form is given as $(\mathfrak{a}, D(\mathfrak{a}))$, where
\begin{align*}
D(\mathfrak{a}) & = \{ u \in H^1(\Omega)\, |\, \exists\, u_0\in H^1(\R^2) : u(x,y,z) = u_0(x,y)\\
& \quad  \mbox{ for almost all } x,y, \in \R, 
0<z <1\}
\end{align*}
and
\begin{align*}
\mathfrak{a}[u,v] &= \int_{\R^2} \left [\partial_x u_0(x,y)\partial_x\bar v_0(x,y) + \partial_y u_0(x,y)\partial_y\bar v_0(x,y) \right ] \,\mathrm{d} (x,y) \\
&\quad + \int_{\R^2} (c(x,y) +d(x,y)) u_0(x,y)\bar v_0 (x,y) \, \mathrm{d} (x,y) .
\end{align*}
\end{thm}

\begin{proof}
Note that $D(\mathfrak{a}_\eps)$ does not depend on $\eps$ and that $\mathfrak{a}_\eps[u] \leq \mathfrak{a}_\delta[u]$ for every $u\in H^1(\Omega)$ whenever $\delta\leq \eps$. This shows that $\mathfrak{a}_\eps$ is increasing as $\eps\downarrow 0$. The estimate
\[
\mathfrak{a}_\eps [u] \geq \eps^{-2}\int_\Omega |\partial_z u|^2\, \di (x,y,z)
\]
shows that $\sup_\eps \mathfrak{a}_\eps[u] <\infty$ implies that $\partial_z u = 0$ almost everywhere. This is equivalent to the fact that $u$ does not depend on $z$, i.e.\ we can find a function $u_0 \in H^1(\R^2)$ such that
$u(x,y,z) = u_0(x,y)$ for almost every $x,y\in \R$ and $0<z<1$. Conversely, it is obvious that for functions with such a representation $\sup_\eps\mathfrak{a}_\eps [u]<\infty$. This proves that the limit form has the claimed domain. Noting that for $u\in D(\mathfrak{a})$,  $\mathfrak{a}_\eps [u] = \mathfrak{a}[u]$ we conclude that also the expression for the limit form itself is clear (by polarization formula).
\end{proof}

We note that $\overline{D(\mathfrak{a})}$ consists exactly of those functions $f\in L^2(\Omega)$ which do not depend on the $z$-variable. This space is isomorphic to $L^2(\R^2)$. The negative of the operator associated to the limit form $\mathfrak{a}$ is of the form \[A= \Delta - (c+d),\] i.e. it is the Laplacian on $L^2(\R^2)$ perturbed by the potential $c+d$. The Simon theorem now reveals that 
\[\lim_{\eps \to 0}\e^{tA_\eps} f = \e^{tA} Pf , \qquad t>0, f \in L^2(\Omega),\]
where $P$, the orthogonal projection on $\overline{D(\mathfrak{a})}$, is given by 
\[ Pu (x,y) = \int_0^1 u (x,y,z) \ud z .\]
To repeat, this result should be interpreted by saying that as $\eps $ tends to zero, mild solutions to \eqref{eq.cauchyproblem} (which are solutions of \eqref{eq.diff1}--\eqref{eq.diff2} seen through a magnifying glass) gradually lose dependence on $z$ and become functions of two variables. Their limit dynamics is then governed by the operator $A$. By our main theorem (Theorem \ref{main}), the same applies to mild solutions to the full reaction-diffusion equations \eqref{eq.cauchyproblemf}, and the master equation for the limit dynamics is 
\[ u'(t) = \Delta u(t) - (c+d)u(t) + PF(u(t)), \qquad t \ge 0 .\] 
The notable difference between this equation and equation \eqref{eq.cauchyproblemf} (or, as a matter of fact, equation \eqref{eq.diff1}) is that, besides the $2D$ Laplace operator  naturally replacing the $3D$ Laplace operator, the right-hand side here has the additional term, equalling $-(c+d)u(t)$, which is a residue of the boundary conditions \eqref{eq.boundary}. Let us recall that in the $3D$ model the latter conditions describe loss of some of the particles touching the boundary. In the $2D$ limit these conditions disappear along with the boundary and their role is taken over by the term just mentioned; this term also describes loss of some of the diffusing particles, but now this process takes place inside the limiting $2D$ region.

\subsection{Modeling fast neurotransmitters}

In this section, we will show that the well-known ODE model of synaptic depression due to Aristizabal and Glavinovi\v{c} \cite{ag} may be obtained as a singular perturbation of the PDE model of Bielecki and Kalita \cite{bielecki}, provided that (as discussed in \cite{bobmor} and \cite{nagrafach}) in the latter model appropriate transmission conditions are introduced. In our previous paper \cite{zmarkusem}, we showed convergence of the related semigroups of operators, but did not tackle the nonlinearity featuring in the master equation.

In the Bielecki and Kalita model, it is assumed that a terminal bouton $\Omega_0$ is subdivided into $3$ subregions
$\Omega_1, \Omega_2$ and $\Omega_3$ which correspond to the so-called immediately available (readily releasable), the small (recycling) and the large (resting) pools, respectively (for discussion of alternative names for these pools see \cite{alabi2012}), with $\Omega_1$ surrounded by $\Omega_2$ which in turn is surrounded by $\Omega_3$. 
 %Given three Balls $B_1, B_2$ and $B_3$ which have a common center and decreasing radii, we can, for example, set $\Omega_0 = B_1$, $\Omega_1 = B_1\setminus B_2$, $\Omega_2 = B_2\setminus B_3$ and
%$\Omega_3 = B_3$.
The master equation for the concentration $u$ of neurotransmitters in $\Omega_0$ is of the form  (see eq. (1) in \cite{bielecki})
\begin{equation} u'(t)= \mathscr L u (t)  + \beta (u^\sharp - u(t))^+. \label{biekal} \end{equation}
Here, $\mathscr{L}$ is a diffusion operator which is complemented with transmission conditions
on $\Gamma_{23}$ (the common boundary of $\Omega_2$ and $\Omega_3$) and $\Gamma_{12}$ (the common boundary of $\Omega_1$ and $\Omega_2$), and with Robin boundary conditions on $\Gamma_{01}$ (the outer boundary  of $\Omega_0$). The transmission conditions require that the co-normal derivative
of the function is proportional to the difference in concentration on both sides of the boundary, and thus allow seeing $\Gamma_{23}$ and $\Gamma_{21}$ as semi-permeable membranes with permeability changing along the boundary; permeability from $\Omega_i$ to $\Omega_j$ and that from $\Omega_j$ and $\Omega_i$ may also differ.  Moreover, $\beta: \Omega_0 \to \R$ is a bounded function which is zero on $\Omega_1\cup\Omega_2$, $u^\sharp$ is a positive constant and $u^+ = \max (u, 0)$.  Thus, the non-linearity is a production term: if $u(t)$ is smaller than $u^\sharp$ at a point of the large pool $\Omega_3$, new neurotransmitters are produced there. Certainly, \eqref{biekal} is an equation of the form \eqref{slin} with  $A=\mathscr L$ and \[ F(t,u) =F(u) = \beta (u^\sharp - u)^+, \qquad u \in L^2(\Omega_0). \] 

In \cite{zmarkusem}, we studied the linearized version of \eqref{biekal} where the nonlinearity $\beta (u^\sharp - u(t))^+$ was replaced by $\beta (u^\sharp - u(t))$. As a particular case of a more general principle of approximating fast diffusions by Markov chains, we proved there that if instead of $\mathscr{L}$ we consider the operator $\mathscr{L}_\kappa$, obtained by replacing the matrix $\mathbb D$ of diffusion coefficients of the operator $\mathscr{L}$ by $\kappa \mathbb D$, the associated contraction semigroups $(\e^{t\mathscr{L}_\kappa})$ converge irregularly, as $\kappa \to \infty$:
\[ \lim_{\kappa \to \infty} \e^{t\mathscr L_\kappa} =\Phi \e^{t Q} \Phi^{-1} P , \qquad t >0.\]
Here $P$ given by
\begin{equation} \label{nowy} Pu = \sum_{i=1}^3 \frac 1{\lam (\Omega_i)} \int_{\Omega_i} u \ud \lam \, \mathds 1_{\Omega_i} , \qquad u \in L^2(\Omega_0), \end{equation}
is the orthogonal projection on the space $\x_0$ of functions $u\in L^2(\Omega_0)$ which are constant on the sets 
$\Omega_1$, $\Omega_2$ and $\Omega_3$, $\lam $ denotes the Lebesgue measure,  and $\Phi$ is the isometric isomorphism between $\C^3$, equipped with the norm 
\[ \|(z_1,z_2,z_3) \|= \sqrt {\sum_{i=1}^3 |z_i|^2 \lam (\Omega_i)}, \]
and $\x_0$, given by
\[
\Phi (u_1, u_2, u_3) = \sum_{j=1}^3 u_j \mathds{1}_{\Omega_j}, \quad (u_1, u_2, u_3) \in \C^3.
\]
Finally, $Q$ is a certain intensity (or: Kolmogorov) matrix, describing a Markov chain with three states corresponding to three aggregated regions $\Omega_1, \Omega_2,\Omega_3$.

As discussed in  
\cite{bobmor} (see also \cite{zmarkusem}, Section 7.2), $Q$, whose entries are obtained by integrating permeability coefficients over the separating membranes,  may be identified with the matrix governing the ODE model of Aristizabal and Glavinovi\v{c}.

Making use of our main theorem, we can now tackle the original equation \eqref{biekal}. Indeed, since $\beta$ is assumed to be bounded, say by a constant $M$, we have 
\[ \| F(u) - F(v) \|_{L^2(\Omega_0)} \le M \|u - v\|_{L^2(\Omega_0)}\]
for any $u,v \in {L^2(\Omega_0)}$.  This means that the non-linear term is Lipschitz continuous. It follows from Theorem \ref{main} that, as $\kappa $ converges to infinity, solutions to \eqref{biekal} (with $\mathscr L$ replaced by $\mathscr L_\kappa$)  converge to an $\x_0$-valued function, which via $\Phi$ defined above may be identified with the $\C^3$-valued solution $u(t)$ of the equation 
\begin{equation}\label{ag} u '(t) = Qu (t) + \Phi^{-1} PF(\Phi (u(t))).  \end{equation} 
 Because of \eqref{nowy}, since $\beta$ is identically zero on $\Omega_1$ and $\Omega_2$, we see that $\Phi^{-1}PF \Phi $ maps a vector $(u_1,u_2,u_3)\in \C^3$ to 
$(0,0,v_3)$ where 
\[ v_3 = \frac 1{\lam (\Omega_3)} \int_{\Omega_3} \beta \ud \lam \, (u^\sharp - u_3)^+ .\]
Therefore, also the non-linearity in the limit equation is a production term: the third coordinate of a solution of \eqref{ag} grows if it happens to be below $u^\sharp$, and the rate of growth is 
 \( \frac 1{\lam (\Omega_3)} \int_{\Omega_3} \beta \ud \lam\).  

%An inspection of the argument presented above shows that although we focused on a model of dynamics of fast neurotransmitters, the same reasoning may be applied to any model   

\vspace{0.2cm}
\textbf {Acknowledgment.}  This research is supported by National Science Center (Poland) grant
2017/25/B/ST1/01804.

%\subsection{Diffusion in a thin layer between a plane and the graph of a smooth function} 

\def\cprime{$'$}\def\cprime{$'$}\def\cprime{$'$}\def\polhk#1{\setbox0=\hbox{#1}{\ooalign{\hidewidth
  \lower1.5ex\hbox{`}\hidewidth\crcr\unhbox0}}}\def\polhk#1{\setbox0=\hbox{#1}{\ooalign{\hidewidth
  \lower1.5ex\hbox{`}\hidewidth\crcr\unhbox0}}}

\end{document}